\documentclass[a4paper,10pt]{amsart}

\usepackage[english]{babel}
\usepackage[utf8]{inputenc}

\usepackage{amssymb}
\usepackage{amsmath}
\usepackage{amsthm}
\usepackage{mathtools}
\usepackage{bbm}

\usepackage{enumerate}
\usepackage{tikz} 
\usepackage{tikz-cd} 
\usepackage{marginnote}
\usepackage{enumitem}
\usepackage{hyperref}
\usepackage{nicefrac}
\usepackage{orcidlink}

\usepackage{csquotes}
\usepackage[style=numeric,isbn=false,url=true,maxnames=50, doi=true, backend = biber]{biblatex}
\AtEveryBibitem{%
  \ifentrytype{thesis}{}{\clearfield{doi}}
  \ifentrytype{book}{\clearfield{url}}{}
  \ifentrytype{article}{\clearfield{url}}{}
  \clearname{translator}%
  \clearname{editorb}%
  \clearfield{pagetotal}%
  \clearfield{editorbtype}%
  \clearfield{month}%
  \clearfield{day}%
  \clearlist{notes}%
  \clearlist{language}%
}%

\DefineBibliographyStrings{english}{%
  edition = {edition},%
}%

\DefineBibliographyStrings{ngerman}{%
  edition = {Edition},%
}%

\usepackage[hyphenbreaks]{breakurl}

\newcommand{\bbC}{\mathbb{C}}
\newcommand{\bbD}{\mathbb{D}}

\newcommand{\bbN}{\mathbb{N}}

\newcommand{\bbR}{\mathbb{R}}

\newcommand{\bbT}{\mathbb{T}}

\newcommand{\bbZ}{\mathbb{Z}}

\DeclareMathOperator{\id}{id}
\DeclareMathOperator{\one}{\mathbbm{1}}

\newcommand{\argument}{\mathord{\,\cdot\,}}

\DeclareMathOperator{\fix}{fix}
\newcommand{\norm}[1]{\left\lVert #1 \right\rVert}
\newcommand{\modulus}[1]{\left\lvert #1 \right\rvert}

\newcommand{\restricted}[1]{|_{#1}}
\newcommand{\ContB}{\mathrm{C}_{\mathrm{b}}}

\renewcommand{\epsilon}{\varepsilon}
\renewcommand{\phi}{\varphi}

% ...

\theoremstyle{definition}
\newtheorem{definition}{Definition}[section]
\newtheorem{remark}[definition]{Remark}

\newtheorem{example}[definition]{Example}

\theoremstyle{plain}
\newtheorem{proposition}[definition]{Proposition}
\newtheorem{lemma}[definition]{Lemma}
\newtheorem{theorem}[definition]{Theorem}
\newtheorem{corollary}[definition]{Corollary}

\numberwithin{equation}{section}

\begin{document}

\title{A Note on the Uniform Ergodicity of Dynamical Systems}
\author{Julian Hölz
\orcidlink{0000-0001-5058-9210}}
\address{University of Wuppertal, Gaußstr. 20, 42119 Wuppertal}
\email{hoelz@uni-wuppertal.de}
%
%
\date{\today}
\begin{abstract}
	We study the uniform ergodicity property for non-invertible topological and measure-preserving dynamical systems.  It is shown that for topological dynamical systems uniform ergodicity is equivalent to eventually periodicity and that for measure preserving systems it is equivalent to periodicity. To obtain our results, we prove a result on the long-term behavior of lattice homomorphisms that have $1$ isolated in its spectrum.
\end{abstract}

\maketitle

\section{Introduction}

For dynamical systems $\varphi : X \to X $, where $X$ is either a topological space and $\varphi$ continuous or $X$ is a probability space and $\varphi$ measure-preserving, we study the associated \emph{composition operator} or \emph{Koopman operator}
\begin{align*}
    T_\varphi : E \to E, \quad f \mapsto f \circ \varphi,
\end{align*}
where $E = \ContB(X)$ or $E = L^p(X)$, $1 \leq p < \infty $, respectively. These operators where introduced by Koopman and von Neumann in~\cite{KoopmanVonNeumann-DynamicalSystemsContSpectra-1932}. This class of operators has proven useful in the study of ergodic properties of dynamical systems (see~\cite{Eisner-OperatorTheoreticAspects-2015}). More recently, composition operator have gained popularity as a tool in the numerical analysis of dynamical systems (see~\cite{MauroyMezic-TheKoopmanOperator-2020}).

The composition operator $T_\varphi$ commutes with the modulus on $E$ in the sense that $\modulus{T_\varphi f} = T_\varphi \modulus{f}$ holds for all $f \in E$. Therefore, composition operators belong to a class of operators called \emph{Banach lattice homomorphisms}. In the following, we obtain a description of the following property for the class of Banach lattice homomorphism with spectral radius $r(T) = 1$: A linear operator $T : E \to E$ is called \emph{uniformly ergodic} if the sequence of Cesàro averages
\begin{align*}
    A_n[T] := \frac{1}{n} \sum_{k = 0}^{n - 1} T^k
\end{align*}
for $n \in \bbN$ converges in operator norm. Then the sequence of Cesàro averages converges to a projection $P:E \to E$ with $PE = \fix(T)$ (see, e.g.,~\cite[Theorem~8.5 on p.~138]{Eisner-OperatorTheoreticAspects-2015}).

Let us review some known results and how this paper improves upon them:
\begin{enumerate}[label =(\roman*), labelindent = 0pt, wide]
    \item
          In~\cite[Proposition~W.12]{DerndingerNagelPalm-ErgodicTheoryPerspective-1987} it is shown that an isometric lattice homomorphism $T$ is uniformly ergodic if and only if it is periodic. We show that a non-isometric uniformly ergodic lattice homomorphism can be split into a periodic part and one that decays exponentially to $0$ (see Theorem~\ref{theorem:uniform-almost-periodicity-isolated-spectral-value} and Corollary~\ref{corollary:almost-uniform-periodicity}).

    \item
          It is well-known that a linear operator $T$ on a Banach space with $1 \in \sigma(T)$ and $\lim_{n \to \infty} \frac{1}{n} \norm{T^n} = 0$ the uniform ergodicity is equivalent to $1$ being a pole of first order of the resolvent of $T$ (see~\cite[Theorem~3.16]{Dunford-SpectralTheoryIConvToProj-1943}). For lattice homomorphisms we prove that this is also equivalent to the formally weaker condition that $1$ is an isolated point in the spectrum of $T$ (see Corollary~\ref{corollary:almost-uniform-periodicity}). Moreover, we show that for lattice homomorphisms the constraint $\lim_{n \to \infty} \frac{1}{n} \norm{T^n} = 0$ can be replaced by the formally weaker assumption $r(T) = 1$.

    \item
          In~\cite[Corollary~2]{IonescuTulcea-RandomSeriesSpectraOfMeasurePreservingTrafos-1963} and \cite[Corollaire~1]{Ciprian-SurLesMesuresSpectralTheorieErgodique-1964} it is shown that the spectrum of the composition operator on $L^2$ of an invertible and aperiodic measure-preserving system coincides with the complex unit circle. Hence, $1$ is not isolated in the spectrum, and thus, the system can not be uniformly ergodic. This observation was later strengthened to almost everywhere pointwise convergence in~\cite{Krengel-SpeedConvergenceErgodicTheorem-1978} for transformations on the unit interval $(0,1)$ endowed with the Lebesgue measure and composition operators on $L^p$ and in~\cite[Propositon~3.2.3]{Petersen-ErgodicTheory-1983} to ergodic systems on non-atomic probability spaces. We show that every measure-preserving dynamical system for which the Cesàro averages have a uniform convergence rate must be periodic (see Theorem~\ref{theorem:Lp-periodic}), thereby complementing the recent results in~\cite[Appendix]{CohenLin-UniformErgodicity-2023}.

    \item
          In~\cite[Satz~5]{Scheffold-ErzeugteAlgebraMarkovVerbandsoperatoren-1971} it is proved that topological dynamical systems on compact Hausdorff spaces that are uniformly ergodic have uniformly bounded orbit lengths. In Theorem~\ref{theorem:eventual-periodicity}, we generalize this to non-compact spaces and show that, in this case, the dynamical system is eventually periodic.
\end{enumerate}

Throughout we let $E$ be a complex Banach lattice and $T: E \to E$ be a lattice homomorphism. We denote by $\bbT$ the unit sphere in $\bbC$ and by $\bbD$ the open unit disk in $\bbC$. For a basic introduction to Banach lattice theory, we refer to~\cite[Chapter~9]{AliprantisBorder-InfiniteDimensionalAnalysis-2006}, \cite{MeyerNieberg-BanachLattices-1991} or~\cite{Schaefer-BanachLattices-1974}. By convention a resolvent of a bounded linear operator $T: X \to X$ on some Banach space $X$ is defined to be $R(\lambda, T) \coloneq {(\lambda - T)}^{-1}$ wherever it exists.

\section{Uniform Almost Periodicity and Uniform Ergodicity}
\label{section:uniform-ergodicity}

It follows from~\cite[Theorem~3.16]{Dunford-SpectralTheoryIConvToProj-1943} that if $T$ is a uniformly ergodic operator on a complex Banach space $E$, then $r(T) \leq 1$ and $1$ is either in the resolvent set of $T$ or isolated in the spectrum of $T$.

Let $E$ now be a complex Banach lattice and $T : E \to E$ be a lattice homomorphism that has spectral radius $r(T)  = 1$. We investigate the long-term behavior of the powers ${(T^n)}_{n \in \bbN}$ in case $1$ is an isolated value of $\sigma(T)$. We recall that an eigenvalue $\lambda \in \bbC$ of a linear operator $T$ is called \emph{semi-simple} if $\ker(\lambda - T) = \ker \left({(\lambda - T)}^n \right)$ for all $n \in \bbN$.

\begin{theorem}\label{theorem:uniform-almost-periodicity-isolated-spectral-value}
	Let $E$ be a complex Banach lattice and $T: E \to E$ be linear and a lattice homomorphism that has spectral radius $r(T) = 1$. If $1$ is an isolated value in the spectrum $\sigma(T)$, then $T$ is power-bounded and uniformly ergodic. Moreover, there exist a closed lattice ideal $I_{\mathrm{stab}}$ and a closed sublattice $E_{\mathrm{per}}$ of $E$ such that $E$ decomposes as
	\begin{align*}
		I_{\mathrm{stab}} \oplus E_{\mathrm{per}} = E
	\end{align*}
	and
	\begin{enumerate}[label = \upshape (\roman*)]
		\item $I_{\mathrm{stab}}$, $E_{\mathrm{per}}$ are $T$-invariant,
		\item $T\restricted{E_{\mathrm{per}}}$ is periodic and
		\item $\lim_{n \to \infty} \norm{{(T\restricted{I_{\mathrm{stab}}})}^n} = 0$.
	\end{enumerate}
	Moreover, the peripheral spectrum $\sigma(T) \cap \mathbb{T}$ consists of a finite union of roots of unity, every $\lambda \in \sigma(T) \cap \mathbb{T}$ is a semi-simple eigenvalue and the equality
	\begin{align} \label{eq:periodic-decomposition}
		E_{\mathrm{per}} = \bigoplus_{\lambda \in \sigma(T) \cap \mathbb{T}} \ker(\lambda - T)
	\end{align}
	holds, i.e., $E_{\mathrm{per}}$ is the sum of eigenspaces of $T$ corresponding to unimodular eigenvalues of $T$.
\end{theorem}

Our primary tool in establishing the proof of Theorem~\ref{theorem:uniform-almost-periodicity-isolated-spectral-value} relies on the observation that the spectrum of a lattice homomorphisms $T$ is cyclic (see~\cite[Theorem~V.4.4 on p.~325]{Schaefer-BanachLattices-1974}), i.e., if $\lambda = \modulus{\lambda} \mathrm{e}^{i\theta} \in \sigma(T)$ for some appropriate phase $\theta \in [0,2\pi)$, then $\modulus{\lambda} \mathrm{e}^{i k \theta} \in \sigma(T)$ for all $k \in \bbZ$. We employ this observation to describe the structure of the peripheral spectrum of $T$.

\begin{lemma}\label{lemma:peripheral-point-spectrum-lattice-homo-isolated-one}
	Let $E$ be a complex Banach lattice and $T: E \to E$ be a Banach lattice homomorphism with $r(T) = 1$ and suppose that $\sigma(T) \cap \bbT$ is a proper subset of $\bbT$. Then the peripheral spectrum $\sigma(T) \cap \bbT$ consists of a finite union of finitely many subgroups of $\bbT$.
	In particular, the peripheral spectrum contains finitely many points.
\end{lemma}

\begin{proof}
    This follows from the cyclicity mentioned before the lemma and the density of $\{\alpha^n : n \in \bbZ \}$ in $\bbT$ for each $\alpha \in \bbT$ that is not a root of unity.
\end{proof}

\begin{proof}[Proof of Theorem~\ref{theorem:uniform-almost-periodicity-isolated-spectral-value}]
    It is discussed in Corollary~\ref{corollary:almost-uniform-periodicity} that the decomposition of the space implies the uniform ergodicity of $T$. The power-boundedness of $T$ also is a direct consequence of the decomposition. So let us first show the decomposition.\medskip
    
	It follows from the cyclicity of the spectrum, discussed before Lemma~\ref{lemma:peripheral-point-spectrum-lattice-homo-isolated-one}, and from the lemma itself that every point in the peripheral spectrum $\sigma(T) \cap \bbT$ must be isolated in $\sigma(T)$.

	It follows that there exists a constant $\rho \in (0, 1)$ such that every $\lambda \in \sigma(T) \cap \bbD$ satisfies $\modulus{\lambda} < \rho$. In other words, the spectral points outside of the peripheral spectrum can be bounded ``uniformly'' away from $\bbT$. In particular, we have $\sigma(T) \cap \rho \cdot \bbT = \emptyset$, and hence, we may define $P \in \mathcal{L}(X)$ to be the spectral projection corresponding to the closed subset $\{ \lambda \in \sigma(T) : \lvert \lambda \rvert < \rho \}$ of the spectrum. It is defined by
	\begin{align*}
		P \coloneq \frac{1}{2 \pi i} \oint_{\lvert \lambda \rvert = \rho} R(\lambda, T) \, \mathrm{d}\lambda,
	\end{align*}
	where $R(\argument, T) = {(\argument - T)}^{-1}$ denotes the resolvent of $T$.
	The projection $P$ decomposes the space $E$ into two closed subspaces as
	\begin{align}\label{eq:decomposition-of-space}
		E = PE \oplus (I-P)E \eqcolon I_{\mathrm{stab}} \oplus E_{\mathrm{per}}.
	\end{align}
	Moreover, the image $I_{\mathrm{stab}} = PE$ is a closed and $T$-invariant linear subspace of $E$ and with $r(T\restricted{I_\mathrm{stab}}) < \rho$ and $\sigma(T\restricted{E_{\mathrm{per}}}) = \sigma(T) \cap \bbT$ (see~\cite[Theorem~VII.3.20]{DunfordSchwartz-LinearOperatorsIGeneralTheory-1958}).\medskip

	Fix $\mu \in (r(T\restricted{I_\mathrm{stab}}), \rho)$. To show that $I_{\mathrm{stab}}$ is an ideal we prove the equality
	\begin{align*}
		I_{\mathrm{stab}} = \{x \in E: \lim_{n \to \infty} \mu^{-n} \lVert T^n x \rVert = 0 \} \eqcolon I.
	\end{align*}

	``$I_{\mathrm{stab}} \subseteq I$'': The inclusion holds, since we have that
	\begin{align*}
		\lim_{n \to \infty} \mu^{-n} \lVert {\left( T\restricted{I_\mathrm{stab}} \right)}^n\rVert = 0
	\end{align*}
    by the choice of $\mu$.

    \smallskip

	``$I \subseteq I_{\mathrm{stab}}$'': Let $x \in I$. Then there exists $C > 0$ such that $\norm{T^n x} \leq C \mu^n$ for all $n \in \bbN$. Hence, the Neumann series of the resolvent $R(\argument, T)x$ converges absolutely for all $\lambda \in \bbC$ with $\modulus{\lambda} > \mu$, i.e., we have
    \begin{align*}
        R(\lambda, T) x = \sum_{n = 0}^\infty \lambda^{-(n + 1)} T^n x
    \end{align*}
    for all $\modulus{\lambda} > \mu$.
    Since $\mu < \rho < 1$, we obtain from Cauchy's integral formula for derivatives that
    \begin{align*}
        Px &= \frac{1}{2 \pi i} \oint_{\modulus{\lambda} = \rho} R(\lambda, T)x \, \mathrm{d}\lambda \\
        &= \sum_{n = 0}^\infty \frac{T^n x}{n!} \left( \frac{n!}{2 \pi i} \oint_{\modulus{\lambda} = \rho} \frac{1}{\lambda^{n + 1}} \, \mathrm{d} \lambda \right) = x.
    \end{align*}
    Hence, $x = Px \in I_{\mathrm{stab}}$ follows.\medskip

    Observe that $I$ is a (lattice) ideal, since $T$ is a lattice homomorphism. As, additionally, $I = I_{\mathrm{stab}}$ is closed, it follows that $E/I_{\mathrm{stab}}$ is a Banach lattice (see~\cite[Propositon~II.11.4~on~p.~137]{Schaefer-BanachLattices-1974}). Since $I_{\mathrm{stab}}$ is $T$-invariant we can define the induced operator of $T$ on $E/I_{\mathrm{stab}}$ by
	\begin{align*}
		T_/ : E/I_{\mathrm{stab}} \to E/I_{\mathrm{stab}}, \qquad f + I_{\mathrm{stab}} \mapsto Tf + I_{\mathrm{stab}}.
	\end{align*}
	Notice that $T_/$ is a lattice homomorphism on $E/I_{\mathrm{stab}}$. From~\eqref{eq:decomposition-of-space} we obtain that the mapping
	\begin{align*}
		S : E/I_{\mathrm{stab}} \to E_{\mathrm{per}}, \qquad f + I_{\mathrm{stab}} \mapsto (I - P) f
	\end{align*}
	is an isomorphism of Banach spaces\footnote{At this stage of the proof we do not know whether $E_{\mathrm{per}}$ is a Banach lattice, so this isomorphism can only be one of Banach spaces.}. Moreover, it follows that $T_/ = S^{-1}T S$, so $T_/$ is conjugated to $T\restricted{E_{\mathrm{per}}}$. Thus, $\sigma(T_/) = \sigma(T\restricted{E_{\mathrm{per}}}) = \sigma(T) \cap \mathbb{T}$. Since, by Lemma~\ref{lemma:peripheral-point-spectrum-lattice-homo-isolated-one}, the peripheral spectrum of $T$ consists of a finite union of roots of unity, there exists a $N \in \bbN$ such that $\sigma\left({(T_/)}^N \right) = \{1\}$. Employing~\cite[Corollary~2.2]{SchaeferWolffArendt-OnLatticeIsoWithPositiveRealSpectrum-1978} or~\cite{Zhang-TwoProofsTheoremSchaefferWolffArendt-1992} we obtain that ${(T_/)}^N$ is the identity operator on $E/I_{\mathrm{stab}}$. From the conjugacy it follows that $T^N$ acts as the identity on $E_{\mathrm{per}}$. Thus $E_{\mathrm{per}}$ is contained in the fixed space $\mathrm{fix}(T^N)$. Since the powers of $T$ converge strongly to $0$ on $I_{\mathrm{stab}}$, it follows from the decomposition~\eqref{eq:decomposition-of-space} that we even have the equality $E_{\mathrm{per}} = \mathrm{fix}(T^N)$. As fixed spaces of lattice homomorphisms are sublattices, it follows that $E_{\mathrm{per}}$ is a sublattice of $E$. Moreover, since ${(T\restricted{E_{\mathrm{per}}})}^N$ is the identity on $E_{\mathrm{per}}$, the operator $T\restricted{E_{\mathrm{per}}}$ is periodic. \medskip

    The decomposition~\eqref{eq:periodic-decomposition} now follows from a standard result from linear algebra, see, e.g.,~\cite[Lemma~2.2.3]{Glueck-InvariantSetsLong-2017} applied to the polynomial $p(X) = X^N - 1$.
\end{proof}

\begin{remark}
\label{rem:information-proof}
	\begin{enumerate}[label = (\roman*)]
		\item \label{it:rem:information-proof:1-is-spectral-value}
		      Note that the assumption $r(T) = 1$ automatically implies that $1 \in \sigma(T)$. This follows easily from the cyclicity of the spectrum discussed before Lemma~\ref{lemma:peripheral-point-spectrum-lattice-homo-isolated-one} and is even true for the more general class of positive operators on Banach lattices (see~\cite[Proposition~4.1.1~i)]{MeyerNieberg-BanachLattices-1991}).

		\item
		      Corollary~2.2 cited from~\cite{SchaeferWolffArendt-OnLatticeIsoWithPositiveRealSpectrum-1978} in the above proof holds for the larger class of Lamperti operators (see~\cite[Corollary~3.6.2]{Arendt-SpectralPropertiesLamperti-1983}). It is part of a class of theorems called Gelfand's $T = \id$ theorems (see, e.g.,~\cite[Theorem B.17]{EngelNagel-Semigroups-2000} or~\cite{Gelfand-TheorieCharaktereAbelscherTopologischenGruppen-1941} for the original result). In~\cite[Lemma~2.1]{LinShoikhetSuciu-RemarksUniformErgodic-2015} a version for a priori uniform ergodic operators on Banach spaces is given.

		\item
		      The fact that $I_{\mathrm{stab}}$ is a closed ideal also follows directly from~\cite[Theorem~4.1]{Arendt-SpectralPropertiesLamperti-1983} and is proved for the more general class of Lamperti operators.
	\end{enumerate}
\end{remark}

Notice that, although $I_{\mathrm{stab}}$ is a closed ideal, in general, it is not a band, as the following example shows.

\begin{example}
	There exists a Banach lattice $E$ and a Banach lattice homomorphism $T: E \to E$ with $r(T) = 1$ that has $1$ as an isolated point in its spectrum but the closed ideal $I_{\mathrm{stab}}$ is not a band.

	Consider the Banach lattice $E \coloneq c(\bbN)$ of all convergent sequences in $\bbC$ endowed with the supremum norm. Define the operator
	\begin{align*}
		T: E \to E, \qquad a = {(a_n)}_{n \in \bbN} \mapsto \lim_{n \to \infty} a_n \cdot \one,
	\end{align*}
	where $\one = {( 1 )}_{n \in \bbN}$. Since $T$ is a projection it follows that $\sigma(T) = {0, 1}$ and that the conditions of Theorem~\ref{theorem:uniform-almost-periodicity-isolated-spectral-value} are satisfied. Since $E_{\mathrm{per}} = \fix(T) = \mathrm{span} \{\one \}$, we have $I_{\mathrm{stab}} = c_0(\bbN)$, which is not a band in $c(\bbN)$.
\end{example}

Let us state four corollaries of Theorem~\ref{theorem:uniform-almost-periodicity-isolated-spectral-value}. In case $T$ is isometric, we arrive at the following conclusion This is~\cite[Proposition~W.12]{DerndingerNagelPalm-ErgodicTheoryPerspective-1987}.

\begin{corollary}\label{corollary:isometry}
    Let $E$ be a complex Banach lattice and $T: E \to E$ be an isometric lattice homomorphism with $r(T) = 1$ that has $1$ as an isolated point in its spectrum. Then $T$ is periodic.
\end{corollary}
\begin{proof}
    It follows from the isometry of $T$ that $I_{\mathrm{stab}} = \{0 \}$.
\end{proof}

The second corollary connects the assumption that $1$ is isolated in $\sigma(T)$ to two a priori distinct notion on the long-term behavior of ${(T^n)}_{n \in \bbN}$. We call a linear operator $T:E \to E$ \emph{uniformly almost periodic} if there exists a periodic operator $S: E \to E$ such that
\begin{align*}
	\lim_{n \to \infty} \lVert T^n - S^n \rVert = 0.
\end{align*}

\begin{corollary}\label{corollary:almost-uniform-periodicity}
	Let $E$ be a complex Banach lattice and $T: E \to E$ be a lattice homomorphism with $r(T) = 1$. Then the following statements are equivalent.
	\begin{enumerate}[label = \upshape (\roman*)]
		\item \label{item:corollary:almost-uniform-periodicity:one-isolated}
		      The point $1$ is isolated in the spectrum of $T$.

		\item \label{item:corollary:almost-uniform-periodicity:uniformly-almost-periodic}
		      The operator $T$ is uniformly almost periodic.

		\item \label{item:corollary:almost-uniform-periodicity:uniformly-mean-ergodic}
		      The operator $T$ is uniformly ergodic.
	\end{enumerate}
\end{corollary}

The uniform ergodicity for the more general classes of positive operators and Markov operators was studied by Michael Lin. We refer to~\cite{Lin-QuasiCompactnessUniformErgodicityPositive-1978} and~\cite{Lin-QuasiCompactnessUniformErgodicityMarkov-1975}, respectively.

\begin{proof}[Proof of Corollary~\ref{corollary:almost-uniform-periodicity}]
	``\ref{item:corollary:almost-uniform-periodicity:one-isolated} $\Rightarrow$ \ref{item:corollary:almost-uniform-periodicity:uniformly-almost-periodic}'':
	This is Theorem~\ref{theorem:uniform-almost-periodicity-isolated-spectral-value}.\smallskip

	``\ref{item:corollary:almost-uniform-periodicity:uniformly-almost-periodic} $\Rightarrow$ \ref{item:corollary:almost-uniform-periodicity:uniformly-mean-ergodic}'':
	Let $N \in \bbN$ be the period of $S$ and define $P \coloneq \lim_{n \to \infty} \frac{1}{n} \sum_{k = 0}^{n - 1} S^k$, which exists since $S$ is periodic. As noted in the introduction, the operator $P$ is then a projection onto $\mathrm{fix}(S)$. By~\ref{item:corollary:almost-uniform-periodicity:uniformly-almost-periodic}, it follows that $\norm{A_n[T] - A_n[S]} \to 0$, so $\norm{A_n[T] - P} \to 0$.\smallskip

	``\ref{item:corollary:almost-uniform-periodicity:uniformly-mean-ergodic} $\Rightarrow$ \ref{item:corollary:almost-uniform-periodicity:one-isolated}'':
	If $T$ is uniformly ergodic, then it follows that $\lim_{n \to \infty} \frac{1}{n} \norm{T^n} = 0$, and thus, by~\cite[Theorem~3.16]{Dunford-SpectralTheoryIConvToProj-1943} and the fact that $1 \in \sigma(T)$ (cf.\ Remark~\ref{rem:information-proof}\ref{it:rem:information-proof:1-is-spectral-value}), it follows that $1$ is a pole of the resolvent $R(\argument, T)$, so it is isolated in $\sigma(T)$.
\end{proof}

When we do not require that $T$ has spectral radius $1$, we obtain the following characterization of uniform ergodicity for lattice homomorphisms.

\begin{corollary}
    Let $E$ be a complex Banach lattice. Then a Banach lattice homomorphism $T:E \to E$ is uniformly ergodic if and only if it is power-bounded and $1$ is either isolated in the spectrum of $T$ or in the resolvent set of $T$.
\end{corollary}
\begin{proof}
    ``$\Rightarrow$'':
    The uniform ergodicity implies that $\lim_{n \to \infty} \frac{1}{n} \norm{T^n} = 0$, which implies $r(T) \leq 1$. In case $r(T) < 1$, then the power-boundedness follows immediately and $1$ is in the resolvent set of $T$. When $r(T) = 1$, then $T$ is power-bounded and has $1$ isolated in the spectrum of $T$ by Corollary~\ref{corollary:almost-uniform-periodicity} and Theorem~\ref{theorem:uniform-almost-periodicity-isolated-spectral-value}.

    ``$\Leftarrow$'':
    The power-boundedness implies that $r(T) \leq 1$. When $r(T) < 1$, then the uniform ergodicity is immediate. If $r(T) = 1$, then $1$ must be in the spectrum of $T$ by the cyclicity of the spectrum mentioned before Lemma~\ref{lemma:peripheral-point-spectrum-lattice-homo-isolated-one}. Thus, by assumption, $1$ is isolated in $\sigma(T)$ and Corollary~\ref{corollary:almost-uniform-periodicity} implies the uniform ergodicity of $T$.
\end{proof}

We can now give an easy spectral characterization of all periodic lattice homomorphisms. Similar results for invertible positive operators can be found in~\cite[Theorem~3.1]{Zhang-AspectsSpectralPositive-1992} or~\cite[Theorem~5.3]{Zhang-SpectralPropertiesPositiveOperators-1993}. For power-bounded operators on reflexive Banach spaces a similar result can be found in~\cite[Corollary~2.7]{Lin-ReflexiveBanachPowerBoundedAlmostPeriodic-2020}.

\begin{corollary}
	Let $E$ be a complex Banach lattice and $T: E \to E$ be a lattice homomorphism. Then $T$ is periodic if and only if its spectrum $\sigma(T)$ is a proper subset of $\bbT$.
\end{corollary}
\begin{proof}
    ``$\Rightarrow$'':
    This is straightforward to see.\smallskip

    ``$\Leftarrow$'':
	If the spectrum of $T$ is a proper subset of $\bbT$, then $r(T) = 1$ follows. Hence, by Lemma~\ref{lemma:peripheral-point-spectrum-lattice-homo-isolated-one} the spectrum is a finite union of finite subgroups of $\bbT$, so $1$ is isolated in the peripheral spectrum $\sigma(T) \cap \bbT$. As by assumption the peripheral spectrum coincides with the spectrum, it follows that $1$ is isolated in $\sigma(T)$.
	Thus, by Theorem~\ref{theorem:uniform-almost-periodicity-isolated-spectral-value} it suffices to prove that $I_{\mathrm{stab}} = \{0\}$.
 
    Recall that in the proof of Theorem~\ref{theorem:uniform-almost-periodicity-isolated-spectral-value}, the ideal $I_{\mathrm{stab}}$ was defined to be the image of the spectral projection $P$ corresponding to the spectral values that are not peripheral. By the assumptions on the spectrum we have $P = 0$ from which the statement follows.
\end{proof}

\section{Application to Composition Operators}
\label{section:composition-operators}

We apply Theorem~\ref{theorem:uniform-almost-periodicity-isolated-spectral-value} to the two types of composition operators that were introduced at the beginning of the paper. We will begin with a characterization of the uniform ergodicity of topological dynamical systems. To this end we introduce AM-spaces with a unit $e$ and study lattice homomorphisms that have $e$ as fixed point. Later in the chapter we focus on composition operators for measure preserving dynamical systems on $L^p$-spaces for $p \in [1, \infty)$.

We call a Banach lattice $E$ an \emph{AM-space}, if it satisfies
\begin{align*}
    \norm{x \vee y} = \norm{x} \vee \norm{y}
\end{align*}
for all $x,y \in E_+$. An element $e \in E_+$ is called unit, if for every $x \in E$ there exists a number $\lambda > 0$ such that $\modulus{x} \leq \lambda e$.
An AM-space $E$ with a unit $e$ can always be equipped with an equivalent norm $\norm{\argument}_e$ satisfying $\norm{e}_e = 1$, called the \emph{gauge norm of $e$}. The gauge norm of $e$ is defined by
\begin{align*}
    \norm{x}_e \coloneq \inf \{ \lambda > 0  : \modulus{x} \leq \lambda e\}
\end{align*}
for all $x \in E$. Therefore, whenever we work with AM-spaces with a unit $e$ we will tacitly assume that $\norm{e} = 1$.

Important examples of AM-spaces are the spaces of bounded and continuous functions on a topological space. For a topological space $X$ we denote by $\ContB(X)$ the space of all bounded complex-valued continuous functions endowed with the supremum norm. It is straightforward to see that $\ContB(X)$ is an AM-space with the constant function $\one$ as a unit.

The next technical lemma brings~\cite[Proposition~2.3(I)]{Kuehner-WhatKoopmanismDoAttractors-2021} to the realm of lattice operators on AM-spaces that leave a unit fixed. It is proved analogously.

\begin{lemma}\label{lemma:kuehner-nilpotency}
    Let $E$ be a complex AM-space with unit $e$. Let $T: E \to E$ be a lattice homomorphism with $T e = e$. Let $I \subseteq E$ be a $T$-invariant and closed lattice ideal such that $\lVert {\left(T\restricted{I}\right)}^n \rVert$ converges to $0$. Then $T\restricted{I}$ is nilpotent.
\end{lemma}
\begin{proof}
    By a representation theorem of Kakutani-Bohnenblust-Krein (see~\cite[Theorem~3.6]{AbramovichAliprantis-InvitationOperator-2002}) we may assume that $E = \ContB(K)$ for some compact Hausdorff space $K$ and $T$ maps the constant function $\one$ to itself.
    
	By~\cite[Theorem~III.9.1]{Schaefer-BanachLattices-1974} there exists a continuous function $\varphi: K \to K$ such that $T = T_\varphi$ with
	\begin{align*}
		T_\varphi : \ContB(K) \to \ContB(K), \quad f \mapsto f \circ \varphi.
	\end{align*}
    Moreover, there exists a closed set $M$ such that $I = \{f \in \ContB(K): \forall x \in M : f(x) = 0 \}$ (see~\cite[Example~1~on~p.~157]{Schaefer-BanachLattices-1974}).

    We show that there exists $n_0 \in \bbN$ such that $\varphi^{n_0}(K) \subseteq M$, since then for every $f \in I$ and $x \in K$ it follows that $T^{n_0} f(x) = f(\varphi^{n_0}(x)) = 0$. This implies the nilpotency. Assume to show a contradiction that for every $n \in \bbN$ there exists $x \in K$ such that $\varphi^n(x) \in K \setminus M$. By Urysohn's lemma there exists a real-valued, non-negative $f \in I$ with $f(\varphi^n(x)) = 1$ and $\norm{f} = 1$. Then
    \begin{align*}
        \norm{{(T\restricted{I})}^n} \geq \norm{{(T\restricted{I})}^n f} \geq f(\varphi^n(x)) = 1.
    \end{align*}
    This contradicts the assumptions of the lemma.
\end{proof}

Combining the statement of Lemma~\ref{lemma:kuehner-nilpotency} with that of Theorem~\ref{theorem:uniform-almost-periodicity-isolated-spectral-value} we obtain the following immediate corollary for AM-spaces. We will call a map $\psi$ of some set into itself \emph{eventually periodic}, if there exists $k, p \in \bbN$ such that
\begin{align*}
    \psi^{k + np} = \psi^k
\end{align*}
for every $n \in \bbN$.

\begin{corollary}\label{corollary:eventual-periodicity}
    Let $E$ be a complex AM-space with unit $e$. Let $T : E \to E$ be a Banach lattice homomorphism with $Te = e$. Suppose that $1$ is isolated in $\sigma(T)$. Then $T$ is eventually periodic.
\end{corollary}
\begin{proof}
    We argue that $r(T) = 1$ so that we may apply Theorem~\ref{theorem:uniform-almost-periodicity-isolated-spectral-value}. This follows from the fact that $\norm{T^n} = \norm{T^n e} = \norm{e}$ for all $n \in \bbN$ (see~\cite[Lemma~3.2]{AbramovichAliprantis-InvitationOperator-2002}).
\end{proof}

Notice that an eventually periodic operator is uniformly almost periodic, so by Corollary~\ref{corollary:almost-uniform-periodicity} the converse implication of Corollary~\ref{corollary:eventual-periodicity} also holds.

Let us look at a simple example illustrating the necessity of the condition $Te = e$ for some unit $e \in E$ in Corollary~\ref{corollary:eventual-periodicity}.

\begin{example}
    \label{ex:no-unit-invariant}
	There exists a complex AM-space $E$ with unit $e$ and a lattice homomorphism $T : E \to E$ with $r(T) = 1$, which is not eventually periodic.
 
    Consider $E \coloneq \bbR^2$ endowed with the supremum norm and the lattice homomorphism $T: E \to E$ given by the matrix
	\begin{align*}
		T =
    \begin{pmatrix}
        \frac{1}{2} & 0 \\
        0 & 1
    \end{pmatrix}.
    \end{align*}
	Then $\sigma(T) = \{\tfrac{1}{2}, 1\}$. Moreover, $I_{\mathrm{stab}} = \{f \in E: f(1) = 0\}$. However, $T$ is not nilpotent on $I_{\mathrm{stab}}$.
\end{example}

For composition operator of continuous dynamical systems on topological spaces Corollary~\ref{corollary:eventual-periodicity} reads as follows. We call a topological space $X$ \emph{completely Hausdorff}, if for every distinct points $x, y \in X$ there exists a continuous function $f : X \to [0,1]$ with $f(x) = 1$ and $f(y) = 0$.
The following result generalizes~\cite[Satz~5]{Scheffold-ErzeugteAlgebraMarkovVerbandsoperatoren-1971} to completely Hausdorff spaces. Also notice that a full description of the spectrum of lattice homomorphisms on $\ContB(K)$ for compact and Hausdorff $K$ can be found in~\cite[Theorem~2.7]{Scheffold-SpektrumVerbandsoperatorenBanachverbaenden-1971}.

\begin{theorem}\label{theorem:eventual-periodicity}
	Let $X$ be a completely Hausdorff space and $\varphi : X \to X$ be continuous and consider the composition operator
    \begin{align*}
        T_\varphi : \ContB(X) \to \ContB(X), \quad f \mapsto f \circ \varphi.
    \end{align*}
    Then the following assertions are equivalent:
    \begin{enumerate}[label = (\roman*)]
        \item \label{it:theorem:eventual-periodicity:isolated}
              The point $1$ is isolated in $\sigma(T_\varphi)$.
        
        \item \label{it:theorem:eventual-periodicity:eventually-periodic}
              The mapping $\varphi$ is eventually periodic.

        \item \label{it:theorem:eventual-periodicity:uniformly-ergodic}
              The operator $T_\varphi$ is uniformly ergodic.
    \end{enumerate}
\end{theorem}
\begin{proof}
    The implications ``\ref{it:theorem:eventual-periodicity:eventually-periodic}~$\Rightarrow$~\ref{it:theorem:eventual-periodicity:uniformly-ergodic}'' and ``\ref{it:theorem:eventual-periodicity:uniformly-ergodic}~$\Rightarrow$~\ref{it:theorem:eventual-periodicity:isolated}'' immediately follow from Corollary~\ref{corollary:almost-uniform-periodicity}.

    ``\ref{it:theorem:eventual-periodicity:isolated}~$\Rightarrow$~\ref{it:theorem:eventual-periodicity:eventually-periodic}'':
    It follows from Corollary~\ref{corollary:eventual-periodicity} that $T_\varphi$ is eventually periodic. Let $p, k \in \bbN$ be such that $T_\varphi^{k + np} = T_\varphi^k$ for all $n \in \bbN$. Fix $x \in X$ and $n \in \bbN$. We show that $\varphi^{k + np}(x) = \varphi^k(x)$. Indeed, let $y \in X \setminus \{\varphi^{k}(x) \}$. Then, by the completely Hausdorff property of $X$, there exists a continuous function $f: X \to [0,1]$ with $f(\varphi^{k}(x)) = 1$ and $f(y) = 0$. Hence, $f(\varphi^{k + np}(x)) = T_\varphi^{k + np} f (x) = T_\varphi^{k} f (x) = f(\varphi^{k}(x)) = 1$, and thus, $\varphi^{k + n p}(x) \neq y$. Since $y$ was arbitrary, it follows that $\varphi^{k + n p}(x) = \varphi^{k}(x)$.
\end{proof}

The next simple example shows that we can not dismiss the completely Hausdorff condition in Theorem~\ref{theorem:eventual-periodicity}. Recall that a topological space is called Urysohn if every two distinct points have open neighborhoods with disjoint closures. In particular a Urysohn space is Hausdorff.

\begin{example}
    There exists a Urysohn space $X$ and a continuous map $\varphi : X \to X$ such that $1$ is isolated in the spectrum of $T_\varphi$ and $\varphi$ is not eventually periodic.

    There exists a (countable) Urysohn space $Y$ on which every continuous and real-valued function is constant (see~\cite[Theorem~2]{Hewitt-OnTwoProblemsOfUrysohn-1946}). Hence it follows that $\ContB(Y) = \ContB(Y; \bbR) \oplus \mathrm{i} \ContB(Y; \bbR)$ only consists of constant functions. Now set $X \coloneq Y^\bbN$ and endow $X$ with the product topology. It is straightforward to see that $X$ is also Urysohn. Moreover, it is simple to check that $\ContB(X)$ also contain only the constant functions.

    Now consider the left shift
    \begin{align*}
        \varphi: X \to X, \quad {(x_n)}_{n \in \bbN} \mapsto {(x_{n + 1})}_{n \in \bbN}
    \end{align*}
    which is continuous and not eventually periodic. Although, because $\ContB(X)$ only contains constant functions, $T_\varphi$ must be the identity operator.
\end{example}

For composition operators of measure-preserving dynamical systems on $L^p$-spaces for $p \in [1, \infty)$ a similar result to Theorem~\ref{theorem:eventual-periodicity} holds, albeit one of its assertion is slightly stronger.

Let $(X, \mu)$ be a probability space. Then we say that a mapping $\varphi: X \to X$ is \emph{$\mu$-periodic} if there exists a period $n \in \bbN$ such that $\mu(\{ x \in X : \varphi^n(x) = x\}) = 1$.

\begin{theorem}
    \label{theorem:Lp-periodic}
	Let $(X, \mu)$ be a probability space and let $\varphi : X \to X$ be a measure preserving mapping. For $1 \leq p < \infty$ consider the composition operator
    \begin{align*}
		T_\varphi : L^p(X, \mu) \to L^p(X, \mu), \qquad f \mapsto f \circ \varphi.
	\end{align*}
    Then the following assertions are equivalent:
    \begin{enumerate}[label = (\roman*)]
        \item \label{it:theorem:Lp-periodic:isolated}
              The point $1$ is isolated in $\sigma(T_\varphi)$.
        
        \item \label{it:theorem:Lp-periodic:periodic}
              The mapping $\varphi$ is $\mu$-periodic.

        \item \label{it:theorem:Lp-periodic:uniformly-ergodic}
              The operator $T_\varphi$ is uniformly ergodic.
    \end{enumerate}
    Moreover, if the assertions~\ref{it:theorem:Lp-periodic:isolated} and~\ref{it:theorem:Lp-periodic:uniformly-ergodic} hold for some $p \in [1, \infty)$, then they hold for every $p \in [1, \infty)$.
\end{theorem}
\begin{proof}
    The implications ``\ref{it:theorem:Lp-periodic:periodic} $\Rightarrow$ \ref{it:theorem:Lp-periodic:uniformly-ergodic}'' and ``\ref{it:theorem:Lp-periodic:uniformly-ergodic} $\Rightarrow$ \ref{it:theorem:Lp-periodic:isolated}'' immediately follow from Corollary~\ref{corollary:almost-uniform-periodicity}.

	``\ref{it:theorem:Lp-periodic:isolated} $\Rightarrow$ \ref{it:theorem:Lp-periodic:periodic}'':
    Since $\varphi$ is assumed to be measure-preserving, it follows that $T_\varphi$ is an isometry. Hence, Corollary~\ref{corollary:isometry} implies that $T_\varphi$ is periodic. This is equivalent to $\varphi$ being $\mu$-periodic.\medskip

    The last statement of the theorem follows from the observation that assertion~\ref{it:theorem:Lp-periodic:periodic} is independent of $p$.
\end{proof}

\begin{remark}
    An alternative proof of the implication ``\ref{it:theorem:Lp-periodic:isolated} $\Rightarrow$ \ref{it:theorem:Lp-periodic:periodic}'' in Theorem~\ref{theorem:Lp-periodic} can be given as follows. By Theorem~\ref{theorem:uniform-almost-periodicity-isolated-spectral-value} there exists $p \in \bbN$ such that the operator sequence ${(T_\varphi^{pn})}_{n \in \bbN}$ converges in operator norm topology to a projection $Q$. Since $\varphi$ is measure-preserving, the constant function $\one$ is a fixed point of the dual operator $T_\varphi'$ and we obtain $Q' \one = \one$. It follows that $Q$ does not vanish on the set of non-zero positive functions in $L^p(X)$, i.e., $Qf \gneqq 0$ whenever $0 \lneqq f \in L^p(X)$. Thus,~\cite[Corollary~5.6]{GerlachGlueck-LowerBoundAsymptoticBehaviorPositiveSemigroups-2018} yields that $T_\varphi^p$ is the identity mapping.
\end{remark}

Theorem~\ref{theorem:Lp-periodic} admits a more abstract formulation, which holds for a larger class of Banach lattices than the $L^p$-spaces. Recall that $L^p$-spaces for $p \in [1, \infty)$ are Banach lattices with an order continuous norm. 
In fact every reflexive Banach lattice has order continuous norm (see~\cite[Section~2.4]{MeyerNieberg-BanachLattices-1991}). Moreover, the $L^p$-norm is strict monontone in the sense that it satisfies $\norm{f}_p < \norm{g}_p$ for every two $0 \leq f, g \in L^p$ with $f \lneqq g$.

A function $\psi: A \to \bbR$ on a subset $A \subseteq E$ is called \emph{strictly monotone} if $x \lneqq y$ implies $\psi(x) < \psi(y)$ for all $x, y \in A$. Moreover, an element $h \in E_+$ is called \emph{quasi-interior}, if the principle ideal $I_h$, defined by
\begin{align*}
	I_h \coloneq \{x \in E : \exists C \geq 0 : \modulus{x} \leq C h\},
\end{align*}
is dense in $E$ with respect to the norm topology.

If $(X, \mu)$ is a $\sigma$-finite measure space and $h \in L^p(X, \mu)$ with $p \in [1, \infty)$, then $h$ is quasi-interior if and only if $h(x) > 0$ for $\mu$-almost every $x \in X$ (see, e.g.,~\cite[Example~1 in Chapter~II.6 on p.~98]{Schaefer-BanachLattices-1974}).

\begin{proposition}\label{proposition:abstract-periodic}
	Let $E$ be complex Banach lattice with order continuous norm and let $T: E \to E$ be a Banach lattice homomorphism with $r(T) = 1$. Assume there exists a strictly monotone function $\psi: E_+ \to \bbR$ such that $\psi(T x) \leq \psi(x)$ for all $x \in E_+$ and that there exists a quasi-interior point $h \in E_+$ such that $Th = h$. If $1$ is isolated in $\sigma(T)$, then $T$ is periodic.
\end{proposition}
\begin{proof}
	By Theorem~\ref{theorem:uniform-almost-periodicity-isolated-spectral-value} the space $E$ decomposes into a $T$-invariant and closed ideal $I_{\mathrm{stab}}$, on which $T^n$ converges uniformly to $0$, and a $T$-invariant and closed sublattice $E_{\mathrm{per}}$, on which $T$ is periodic. So it suffices to show that $I_{\mathrm{stab}} = \{0\}$.

	Since $E$ has order continuous norm the closed ideal $I_{\mathrm{stab}}$ is a projection band (see~\cite[Corollary~2.4.4]{MeyerNieberg-BanachLattices-1991}); hence, there exists a band projection $P: E \to E$ onto $I_{\mathrm{stab}}$. Since band projections satisfy $0 \leq P \leq I$, we have $Ph \leq h$, and thus, $TPh \leq Th \leq h$. As $TPh \in I_{\mathrm{stab}}$, this implies $TPh = PTPh \leq Ph$, so $Ph$ is a so-called subfix point of $T$.

    Set $g \coloneq h - Ph \geq 0$. Since $Th = h$, we have $0 \leq Ph - TPh = Tg - g$ shows that $Tg \geq g$. If $Tg \neq g$, then $\psi(Tg) > \psi(g)$, which contradicts the assumptions. Hence, we obtain that $Tg = g$, and thus, $TPh = Ph$. Therefore, $Ph \in I_{\mathrm{stab}} \cap E_{\mathrm{per}}$, which implies that $Ph = 0$. Since $h$ is quasi-interior it follows that $P = 0$. This shows that $I_{\mathrm{stab}} = \{ 0 \}$, as claimed.
\end{proof}

\begin{remark}
\begin{enumerate}[label = (\roman*)]
    \item
          In Proposition~\ref{proposition:abstract-periodic} we may replace the boundedness of $T$ with respect to the strictly monotone function $\psi$ by the following weaker condition: there exists a strictly monotone function $\psi: E_+ \to \bbR$ and an integer $n \in \bbN$ such that $\psi(T^n x) \leq \psi(x)$ for all $x \in E_+$.
          With this weaker condition, we can infer from Proposition~\ref{proposition:abstract-periodic} that $T^n$ is periodic, which then implies that $T$ must have been periodic in the first place.

    \item
          The assumption $Th=h$ is necessary for the statement of Proposition~\ref{proposition:abstract-periodic} to be correct. Notice that Example~\ref{ex:no-unit-invariant}, when $E$ is endowed with the $L^1$-norm, yields an lattice homomorphism $T : E \to E$ with $Th \lneqq h$ for every quasi-interior point $h \in E_+$, $\sigma(T) = \{\tfrac{1}{2}, 1\}$, $\norm{T} \leq 1$. However, the assertion of Proposition~\ref{proposition:abstract-periodic} does not hold true.
\end{enumerate}
\end{remark}

The next two examples show that the existence of a strictly monotone function $\psi: E_+ \to \bbR$ with $\psi(Tx) \leq \psi(x)$ can not be dropped in Proposition~\ref{proposition:abstract-periodic}.

\begin{example}
	There exists a measure space $X$, a lattice homomorphism $T: L^1(X) \to L^1(X)$ with $Th = h$ and $r(T) = 1$ that has $1$ isolated in $\sigma(T)$, satisfies $\norm{T} > 1$ and is not periodic. We follow~\cite[Ex.~after Thm.~4.1]{Arendt-SpectralPropertiesLamperti-1983}.

	Indeed, consider $X = \{0,1\}$ with $\mu$ being the counting measure, the constant map $\varphi \equiv 0$. Notice that $\varphi$ is not measure preserving, so Theorem~\ref{theorem:Lp-periodic} is not applicable. We consider the composition operator
    \begin{align*}
        T : L^1(X,\mu) \to L^1(X, \mu), \quad f \mapsto f \circ \varphi.
    \end{align*}
    Identifying $L^1(X, \mu) \cong (\bbR^2, {\lVert \argument \rVert}_1)$ we can replace $T$ with a matrix $T \in \bbR^{2 \times 2}$ given by
	\begin{align*}
		T =
		\begin{pmatrix}
			1 & 0 \\
			1 & 0
		\end{pmatrix}.
	\end{align*}
	Then $T$ is a lattice homomorphisms with $\sigma(T) = \{0, 1\}$.	Moreover, notice that $\lVert T \rVert = 2$ and $h = (1,1)$ is a quasi-interior point with $Th = h$. However, $T$ is not periodic.
\end{example}

Notice that, in the above example, $T^n = T$ for every $n > 1$, so $T$ is eventually periodic. This is an artifact of its finite-dimensionality. It is not true in general, as the next example demonstrates.

\begin{example}
	There exists a measure space $X$, a lattice homomorphism $T: L^1(X) \to L^1(X)$ with $Th = h$ and $r(T) = 1$ for some quasi-interior point $h \in E_+$ that has $1$ isolated in $\sigma(T)$ and is not (eventually) periodic. Moreover, for every strictly monotone function $\psi: {L^1(X)}_+ \to \bbR$ and every $n \in \bbN$ there exists $f \in {L^1(X)}_+$ such that $\psi(T^n f) > \psi(f)$.

	Let $X = [0, 2]$, let $\mu$ be the Lebesgue measure on $X$. Consider the map
	\begin{align*}
		\varphi: X \to X, \quad x \mapsto
		\begin{cases}
			2x, & x \leq 1, \\
			x,  & x \geq 1.
		\end{cases}
	\end{align*}
    Notice that $\varphi$ is not measure preserving, so Theorem~\ref{theorem:Lp-periodic} is not applicable.
    
	The Banach lattice $E \coloneq L^1(X, \mu)$ has order continuous norm, contains the constant function $\one$ as a quasi-interior point and the map
	\begin{align*}
		T : E \to E, \quad f \mapsto f \circ \varphi.
	\end{align*}
	is a lattice homomorphism with $T \one = \one$.
 
	Consider the band 
	\begin{align*}
		I \coloneq \{f \cdot \one_{[0,1]} : f \in L^1\}
	\end{align*}
    and notice that $I$ is $T$-invariant, since for $f \in I$ we have $Tf(x) = 0$ for almost all $x \in [\tfrac{1}{2},2]$, so $T^n \one_{[0,1]} = \one_{\left[0, \nicefrac{1}{2^n} \right]}$ for all $n \in \bbN$. Moreover, for every $f \in I$ we obtain
	\begin{align*}
		\lVert Tf \rVert = \int_0^1 \lvert Tf(x) \rvert \, \mathrm{d} x = \int_0^1 \lvert f(2x) \rvert \, \mathrm{d} x = \frac{1}{2} \int_0^2 \lvert f(x) \rvert \, \mathrm{d} x = \frac{1}{2} \lVert f \rVert.
	\end{align*}
	Hence, $r(T\restricted{I}) \leq \tfrac{1}{2}$.

    Since $I$ is closed and $T$-invariant, we can consider the quotient operator
    \begin{align*}
        T_/ : E/I \to E/I, \quad f + I \mapsto Tf + I.
    \end{align*}
    Then it is straightforward to see that $T_/$ acts as the identity operator on $E/I$. Hence, $\sigma(T_/) = \{1\}$.

    It is a standard result in spectral theory (see, e.g., \cite[Proposition~VI.2.15]{EngelNagel-Semigroups-2000}) that $\sigma(T) \subseteq \sigma(T\restricted{I}) \cup \sigma(T_/)$. Hence, $r(T) = 1$ and $1$ is isolated in $\sigma(T)$. In particular, Theorem~\ref{theorem:uniform-almost-periodicity-isolated-spectral-value} is applicable. It is straightforward to show that $I_{\mathrm{stab}} = I$ and $E_{\mathrm{per}} \cong E/I$.

    Notice that $T^n \one_{[1,2]} = \one_{[\nicefrac{1}{2^n}, 2]}$ for all $n \in \bbN$. This shows that $T$ is not (eventually) periodic. The indicator function $f = \one_{[1,2]}$ also yields an example of a function satisfying $\psi(T^n f) > \psi(f)$ for every strictly monotone function $\psi : E_+ \to \bbR$ and every $n \in \bbN$.
\end{example}

\printbibliography%

\end{document}